\documentclass[twoside]{amsart}
\usepackage{latexsym}
\usepackage{amssymb,amsmath,amsopn}
\usepackage[dvips]{graphicx}   
\usepackage{color,epsfig}      

\newtheorem{thm}{Theorem}
\newtheorem{lem}{Lemma}
\theoremstyle{definition}
\newtheorem{defn}{Definition}
\newtheorem{rem}{Remark}

\newtheorem{prob}{Problem}

\renewcommand{\Re}{\mathbb R}

\newcommand{\BB}{\mathbf B}

\renewcommand{\S}{\mathbb{S}}
\newcommand{\HH}{\mathbb{H}}
\newcommand{\V}{\mathcal{V}}

\def\bea{\begin{eqnarray}}
\def\eea{\end{eqnarray}}

\DeclareMathOperator{\grad}{grad}

\DeclareMathOperator{\inter}{int}
\DeclareMathOperator{\bd}{bd}

\DeclareMathOperator{\conv}{conv}

\DeclareMathOperator{\arsh}{arsinh}
\DeclareMathOperator{\rank}{rank}
\DeclareMathOperator{\sh}{sh}

\DeclareMathOperator{\cc}{cc}
\DeclareMathOperator{\IC}{IC}
\DeclareMathOperator{\cm}{cm}
\DeclareMathOperator{\ccm}{ccm}
\DeclareMathOperator{\skel}{skel}
\DeclareMathOperator{\csch}{csch}

\begin{document}

\title[Centering Koebe polyhedra]{Centering Koebe polyhedra via M\"obius transformations}
\author[Z. L\'angi] {Zsolt L\'angi}
\address{Zsolt L\'angi, MTA-BME Morphodynamics Research Group and Dept. of Geometry, Budapest University of Technology,
Egry J\'ozsef utca 1., Budapest, Hungary, 1111}
\email{zlangi @math.bme.hu}
\thanks{Partially supported by the NKFIH Hungarian Research Fund grant 119245, by the grant BME FIKP-V\'IZ and the UNKP-18-4 New National Excellence Program of EMMI and by the Bolyai Research Scholarship of the Hungarian Academy of Sciences.}


\begin{abstract}
A variant of the Circle Packing Theorem states that the combinatorial class of any convex polyhedron contains elements midscribed to the unit sphere centered at the origin, and that these representatives are unique up to M\"obius transformations of the sphere. Motivated by this result, various papers investigate the problem of centering spherical configurations under M\"obius transformations. In particular, Springborn proved that for any discrete point set on the sphere there is a M\"obius transformation that maps it into a set whose barycenter is the origin, which implies that the combinatorial class of any convex polyhedron contains an element midsribed to a sphere with the additional property that the barycenter of the points of tangency is the center of the sphere. This result was strengthened by Baden, Krane and Kazhdan who showed that the same idea works for any reasonably nice measure defined on the sphere. The aim of the paper is to show that Springborn's statement remains true if we replace the barycenter of the tangency points by many other polyhedron centers. The proof is based on the investigation of the topological properties of the integral curves of certain vector fields defined in hyperbolic space. We also show that most centers of Koebe polyhedra cannot be obtained as the center of a suitable measure defined on the sphere.
\end{abstract}
\maketitle

\section{Introduction}\label{sec:intro}

The famous Circle Packing Theorem \cite{PachAgarwal} states that every simple, connected plane graph can be realized as the intersection graph of a circle packing in the Euclidean plane, or equivalently, on the sphere; that is, by a graph whose vertices are the centers of some mutually nonoverlapping circles, and two vertices are connected if the corresponding circles are tangent.

This theorem was first proved by Koebe \cite{Koebe}, and was later rediscovered by Thurston \cite{Thurston}, who noted that this result also follows from the work of Andreev \cite{Andreev1, Andreev2}.
The theorem has induced a significant interest in circle packings in many different settings, and has been generalized in many directions.
One of the most known variants is due to Brightwell and Scheinerman \cite{Brightwell}. By this result, any polyhedral graph (i.e. any simple, $3$-connected planar graph \cite{Steinitz1, Steinitz2}), together with its dual graph, can be realized simultaneously as intersection graphs of two circle packings with the property that each point of tangency belongs to two pairs of tangent circles which are orthogonal to each other.
Such a pair of families of circles on the unit sphere $\S^2$ centered at the origin $o$ generate a convex polyhedron \emph{midscribed} to the sphere; that is, having all edges tangent to it. In this polyhedron, members of one family, called \emph{face circles}, are the incircles of the faces of the polyhedron, and members of the other family, called \emph{vertex circles}, are circles passing through all edges starting at a given vertex.
This yields the following theorem \cite{Brightwell, Schramm}.

\begin{thm}\label{thm:Koebe}
The combinatorial class of every convex polyhedron has a representative midscribed to the unit sphere $\S^2$.
\end{thm}

Such representatives of combinatorial classes are called Koebe polyhedra. By Mostow's rigidity theorem \cite{Mostow, Gromov}, these representations are unique up to
M\"obius transformations of the sphere.
We note that by a famous result of Steinitz \cite{Steinitz3}, not all combinatorial classes can be represented by polyhedra circumscribed about (or inscribed in) a sphere; in his seminal paper Rivin \cite{Rivin} gave a characterization of the possible classes.

In \cite{Mani}, Mani strengthened this result by showing that up to Euclidean isometries, every combinatorial class can be uniquely represented by a polyhedron midscribed to $\S^2$ such that the barycenter of the tangency points is the origin (cf. also \cite[p.118]{Ziegler} and \cite[p.296a]{Grunbaum}).
Springborn \cite{Springborn} gave an elegant different proof of the same statement, based on the application of the following theorem.

\begin{thm}[Springborn]\label{thm:Springborn}
For any mutually distinct points $v_1,v_2,\ldots, v_n$ on the $d$-dimensional unit sphere $\S^d$ centered at the origin $o$, where $n \geq 3$ and $d \geq 2$, there is a M\"obius transformation $T$ of $\S^d$ such that $\sum_{i=1}^n T(v_i) = o$. Furthermore, if $\tilde{T}$ is another such M\"obius transformation, then $\tilde{T} = RT$, where $R$ is an isometry of $\S^d$.
\end{thm}

Baden, Krane and Kazhdan examined this problem in a more general form \cite{BCK}, and showed that the idea of the proof of Theorem~\ref{thm:Springborn} in \cite{Springborn} can be extended to the center of mass of any sufficiently well-behaved density function on $\S^d$.
Similar problems are investigated in \cite{BernEppstein}, where the authors considered the algorithmic aspects of optimization of circle families on $\S^2$ via
M\"obius transformations. We note that all these results investigate the problem of centering certain configurations on the sphere $\S^d$ (in particular, $\S^2$) via M\"obius transformations, having applications e.g. in computer graphics \cite{BernEppstein, BCK}.

The aim of this paper is to approach this question from geometric point of view, and to examine the problem of centering \emph{Koebe polyhedra} via M\"obius transformations, using various notions of `centers' of polyhedra from the literature. Whereas this problem seems similar to the one investigated in the papers mentioned above, it is worth noting that it cannot be reduced to the investigation of suitable density functions defined on $\S^2$: in Remark~\ref{rem:notdensity} we show that for most notions of centers examined in this paper, every combinatorial class contains a Koebe polyhedron whose center is outside the unit ball $\BB^3$.
In the first part of the paper we show that, apart from uniqueness, Springborn's statement can be generalized for most notions of polyhedron centers appearing in the literature. In addition, we prove a variant of Theorem~\ref{thm:Springborn} for families of circles.

We remark that the variant of Theorem~\ref{thm:main} with respect to the center of mass of the polyhedron, which we state as Problem~\ref{prob:balancing}, proves that every combinatorial class has a representative whose every face, vertex and edge contains a static equilibrium point \cite{balancing}. An affirmative answer to the problem, with many applications in mechanics \cite{balancing, MorseSmale, Holmes}, would be a discrete version of Theorem 1 in \cite{MorseSmale}, stating that for every $3$-colored quadrangulation $Q$ of $\S^2$ there is a convex body $K$ whose Morse-Smale graph, with respect to its center of mass, is isomorphic to $Q$.
These papers also describe possible applications of our problem in various fields of science, from physics to chemistry to manufacturing.

To state our main results, for any convex polyhedron $P \subset \Re^3$, by $\cc(P)$, $\IC(P)$ and for $k=0,1,2,3$ by $\cm_k(P)$ we denote the center of the (unique) smallest ball containing $P$, the set of the centers of the largest balls contained in $P$, and the center of mass of the $k$-dimensional skeleton of $P$, respectively.
Furthermore, if $P$ is simplicial, by $\ccm(P)$ we denote the circumcenter of mass of $P$ (see, e.g. \cite{Tabachnikov}, or Definition~\ref{defn:ccm} in Section~\ref{sec:prelim}).

Our main theorems are the following, where, with a little abuse of notation, if $P$ is a Koebe polyhedron and $T$ is a M\"obius transformation, by $T(P)$ we mean
the polyhedron defined by the images of the face circles and the vertex circles of $P$ under $T$.

\begin{thm}\label{thm:main}
Let $P$ be a Koebe polyhedron, and let $g(\cdot) \in \{ \cc(\cdot), \cm_0(\cdot), \cm_1(\cdot), \cm_2(\cdot) \}$. Then there is some M\"obius transformation $T_g$ such that $g(T_g(P))= o$. Furthermore, there is a M\"obius transformation $T_{\mathrm{ic}}$ with $o \in \IC(T_{\mathrm{ic}}(P))$, and if $P$ is simplicial, then for every $\lambda \in [0,1)$, there is a  M\"obius transformation $T_{\lambda}$ satisfying $\lambda \cm_3(T_{\lambda}(P))+ (1-\lambda)\ccm(T_{\lambda}(P))=o$.
\end{thm}

We remark that even though it seems difficult to state formally under which conditions the method of the proof of Theorem~\ref{thm:main} works, it is plausible to assume that our key lemma, Lemma~\ref{lem:integralcurves}, holds for a rather large class of hyperbolic vector fields, implying the statement of Theorem~\ref{thm:main} for many other possible notions of polyhedron centers.

In the next theorem, which can be regarded as a generalization of Theorem~\ref{thm:Springborn}, by a spherical cap on $\S^d$ we mean a $d$-dimensional closed spherical ball of spherical radius $0 < \rho < \frac{\pi}{2}$. Furthermore, if $T : \S^d \to \S^d$ is a M\"obius transformation, then by $\rho_T(C)$ and $c_T(C)$ we denote the center and the spherical radius of the spherical cap $T(C)$, respectively.  

\begin{thm}\label{thm:circles}
Let $C_1, C_2, \ldots, C_n \subset \S^d$ be spherical caps such that the union of their interiors is disconnected. For $i=1,2,\ldots,n$, let $w_i:\left(0, \frac{\pi}{2} \right) \to (0,\infty)$ be $C^{\infty}$-class functions satisfying $\lim_{t \to \frac{\pi}{2}-0} w_i(t) = \infty$ for all values of $i$.
For any point $q \in \S^d$, let $I(q)$ denote the set of the indices of the spherical caps whose boundary contains $q$, and assume that for any
$q \in \S^d$, we have
\begin{equation}\label{eq:circle_condition}
\lim_{t \to \frac{\pi}{2}-0} \sum_{i \in I(q)} w_i(t) \cos t < \lim_{t \to 0+0} \sum_{i \notin I(q)} w_i(t).
\end{equation}
Then there is a M\"obius transformation $T : \S^d \to \S^d$ such that
\begin{equation}\label{eq:vertices}
\sum_{i=1^n} w_i(\rho_T(C_i)) c_T(C_i) = o.
\end{equation}
\end{thm}

In Section~\ref{sec:prelim}, we introduce our notation and the concepts in our theorems. In addition, we decribe the method of the proofs, and collect some observations that we are going to use. As the simplest case, we prove Theorem~\ref{thm:main} for $\cc(\cdot)$ and $\IC(\cdot)$ in Section~\ref{sec:spheres}.
In Section~\ref{sec:integralcurves} we prove our key lemma, Lemma~\ref{lem:integralcurves}, which is necessary to prove the rest of the cases in Theorem~\ref{thm:main}.
In Section~\ref{sec:main} we prove Theorem~\ref{thm:main} for all the remaining cases based on this lemma.
We continue with the proof of Theorem~\ref{thm:circles} in Section~\ref{sec:circles}.
Finally, in Section~\ref{sec:remarks} we collect some additional remarks and questions.
We note that some elements of the proof can be found in \cite{Springborn}.

\section{Preliminaries}\label{sec:prelim}

\subsection{Polyhedron centers}\label{subsec:centers}

Let $P$ be a convex polyhedron in the Euclidean $3$-space $\Re^3$. For $k=0,1,2,3$, let $\skel_k(P)$ denote the $k$-skeleton of $P$.
Then the \emph{center of mass of $\skel_k(P)$} is defined in the usual way as
\[
\cm_k(P) = \frac{\int_{p \in \skel_k(P)} p \, dv_k}{\int_{p \in \skel_k(P)}\, dv_k},
\]
where $v_k$ denotes $k$-dimensional Lebesgue measure.

The next concept was defined for polygons in \cite{Adler} and for simplicial polytopes in \cite{Tabachnikov} (see also \cite{Akopyan}).
Before introducing it, we point out that the circumcenter of a nondegenerate simplex is the center of the unique sphere containing all vertices of the simplex, and thus, it may be different from the center of the smallest ball containing the simplex.

\begin{defn}\label{defn:ccm}
Let $P$ be an oriented simplicial polytope, and let $o$ be a given reference point not contained in any of the facet hyperplanes of $P$. Triangulate $P$ by simplices whose bases are the facets of $P$ and whose apex is $o$. Let $p_i$ and $m_i$ denote, respectively, the circumcenter and the volume of the $i$th such simplex.
Then the \emph{circumcenter of mass} of $P$ is defined as
\[
\ccm(P) = \frac{\sum_i m_i p_i}{ \sum_i m_i} .
\]
\end{defn}

The authors of \cite{Tabachnikov} show that the circumcenter of mass of a simplicial polytope $P$ is
\begin{itemize}
\item independent of the choice of the reference point, 
\item remains invariant under triangulations of $P$ if no new vertex is chosen from the boundary of $P$.
\item satisfies Archimedes' Lemma: if we decompose $P$ into two simplicial polytopes $Q_1$ and $Q_2$ in a suitable way, then $\ccm(P)$ is the weighted average of $\ccm(Q_1)$ and $\ccm(Q_2)$, where the weights are the volumes of $Q_1$ and $Q_2$,
\item if $P$ is inscribed in a sphere, then its circumcenter of mass coincides with its circumcenter.
\end{itemize}
In addition, they use this point to define the \emph{Euler line} of a simplicial polytope as the affine hull of the center of mass of $P$ and $\ccm(P)$. This definition is a generalization of the same concept defined for simplices. They show that for polygons, any notion of `center' satisfying some elementary properties (i.e. it depends analytically on the vertices of the polygon, commutes with dilatations and satisfies Archimedes's lemma) is necessarily a point of the Euler line.

\subsection{Idea of the proof}\label{subsec:idea}

In the following, let $P$ be a Koebe polyhedron.
The centers and the spherical radii of the vertex circles of $P$ are denoted by $v_i \in \S^2$ and $\alpha_i$, respectively, where $i=1,2,\ldots,n$, and the centers and the radii of its face circles by $f_j \in \S^2$ and $\beta_j$, respectively, where $j=1,2,\ldots,m$.
We note that by \cite{Springborn}, we may assume that the barycenter of the tangency points of $P$ is the origin $o$, implying that $o$ is contained both in $P$ and in its dual, or in other words, the radii of all vertex and face circles of $P$ are less than $\frac{\pi}{2}$.
Thus, for any vertex or face circle there is an associated spherical cap, obtained as the union of the circle and its interior.

In the proof, we regard the sphere $\S^2$ (or, in the proof of Theorem~\ref{thm:circles}, $\S^d$) as the set of the ideal points in the Poincar\'e ball model of the hyperbolic space $\HH^3$ (or $\HH^{d+1}$). Thus, every circle on $\S^2$ is associated to a hyperbolic plane, and every spherical cap is associated to a closed hyperbolic half space.
We note that since the Poincar\'e ball model is conformal, the dihedral angle between two circles on $\S^2$ is equal to the dihedral angle between the two corresponding hyperbolic planes (cf. \cite[Observation 0.1]{HuangLiu}).

For the vertex circle with center $v_i$ we denote the corresponding hyperbolic plane by $V_i$ and the associated closed half space by $\bar{V}_i$.
Similarly, the hyperbolic plane corresponding to the face circle with center $f_j$ is denoted by $F_j$, and the associated closed half space by $\bar{F}_j$.
We set $D = \HH^3 \setminus \left( \left( \bigcup_{i=1}^n \bar{V}_i \right) \cup \left( \bigcup_{j=1}^m \bar{F}_j  \right) \right)$.
Observe that as the radii of all vertex and face circles of $P$ are less than $\frac{\pi}{2}$, we have $o \in D$, and thus, $D$ is a nonempty, open convex set in $\HH^3$.

Let $p \in D \subset \HH^3$ be a point. For any plane $V_i$, consider the geodesic line through $p$ and perpendicular to $V_i$.
Let $v_i(p) \in T_p \HH^3$ denote the unit tangent vector of this line at $p$, pointing towards $V_i$, and let $d_i^v(p)$ denote the hyperbolic distance of $p$ from $V_i$.
We define $f_j(p)$ and $d_j^f(p)$ similarly for the plane $F_j$.

An important point of the proof is the following simple observation. Recall that the \emph{angle of parallelism} of a point $p$ not lying on a hyperbolic hyperplane $H$ is the hyperbolic half angle of the hyperbolic cone with apex $p$ formed by the half lines starting at $p$ and parallel to $H$. Thus, Remark~\ref{rem:parallelism} is a consequence of the fact that the Poincar\'e ball model is conformal, and of a well-known hyperbolic formula \cite{Kerekjarto}. Even though we state it for the $3$-dimensional space $\HH^3$, it also holds in any dimensions.

\begin{rem}\label{rem:parallelism}
Let $H$ be a hyperbolic plane in $\HH^3$ whose set of ideal points is a circle $C$ on $\S^2$ with spherical radius $\alpha$. Then $\alpha$ is the angle of parallelism of $H$ from the origin $o$ (cf. Figure~\ref{fig:parallelism}). In particular, $\cos \alpha = \tanh d$, where $d$ is the hyperbolic distance between $H$ and $o$.
\end{rem}

\begin{figure}[ht]
\begin{center}
\includegraphics[width=0.4\textwidth]{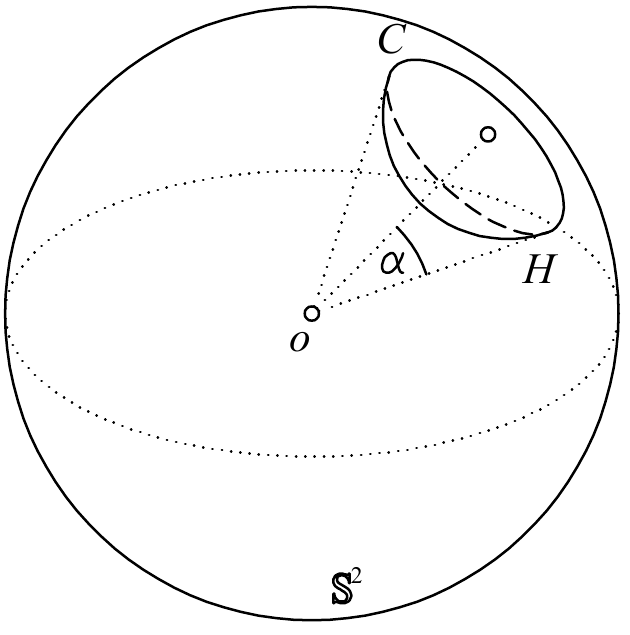}
\caption{The angle of parallelism from the origin of the model is the spherical radius of the circle $C$}
\label{fig:parallelism}
\end{center}
\end{figure}

Among other things, it follows by Remark~\ref{rem:parallelism} that
\begin{equation}\label{eq:angledistance}
\tanh d_i^v(o) = \cos \alpha_i \hbox{ and } \tanh d_j^f(o) = \cos \beta_j \hbox{ for all values of } i, j.
\end{equation}
Furthermore, the metric tensor of the Poincar\'e ball model yields (cf. \cite{Springborn}) that
\begin{equation}\label{eq:gradient}
v_i(o) = \frac{1}{2} v_i \hbox{ and } f_j(o) = \frac{1}{2} f_j \hbox{ for all values of } i,j.
\end{equation}

The idea of the proof of Theorem~\ref{thm:main} in most cases is as follows.
Let $g(\cdot)$ be one of the points in Theorem~\ref{thm:main}.
First, we compute $g(P)$ in terms of the radii and the centers of its vertex and face circles; that is, in a form
$g(P)=\sum_{i=1}^n w_i u_i + \sum_{j=1}^m W_j v_j$, where the coefficients $w_i$ and $W_j$ are smooth functions depending on the values $0 < \alpha_i , \beta_j < \frac{\pi}{2}$. 
Applying the formulas in (\ref{eq:angledistance}) to the coefficients $w_i$ and $W_j$, we obtain a smooth hyperbolic vector field $h : D \to TD$.
Since in this model M\"obius transformations on $\S^2$ are associated to hyperbolic isometries of $\HH^3$, this function has the property that if $T$ corresponds to a hyperbolic isometry that maps $p$ into $o$, then $h(p)= g(T(P))$ for all $p \in D$.
It is well-known that hyperbolic isometries act transitively on $\HH^3$. Thus, to prove the existence of a suitable M\"obius transformation it suffices to prove that 
$h(p) = o_p$ for some $p \in D$.
In the cases of $\cc(\cdot)$ and $\IC(\cdot)$ the function $h$ is not $C^{\infty}$-class; here we use similar, geometric arguments.
In the remaining cases $h$ is smooth; here we examine the properties of the integral curves of $h$.
To prove Theorem~\ref{thm:circles}, we use an analogous consideration.

\subsection{Basic tools}\label{subsec:tools}

In the proof we often use the following geometric observation.

\begin{rem}\label{rem:fv}
For $i=1,2,\ldots, n$ and $j=1,2,\ldots,m$, the $i$h vertex of $P$ is $\frac{v_i}{\cos \alpha_i}$, and the incenter of the $j$th face of $P$ is $\cos \beta_j f_j$.
\end{rem}

Most of the computations will be carried out in the Poincar\'e half space model.

In this model, we regard $\HH^3$ embedded in $\Re^3$ as the open half space $\{ z > 0 \}$. Hyperbolic planes having the `point at infinity' as an ideal point are represented in this model by the intersections of the Euclidean half space $\{ z > 0 \}$ with Euclidean planes parallel to the $z$-axis, we call these hyperbolic planes \emph{vertical}. Hyperbolic planes not having the `point at infinity' as an ideal point are represented by open hemispheres in the Euclidean half space $\{ z > 0 \}$, with their centers on the Euclidean plane $\{z=0\}$, we call these planes \emph{spherical}.
For any plane $H$ in this model, we denote the set of its ideal points, different from the point at infinity, by $H^*$.
We use the same terminology and notation for this model in any dimension.

The last remark in this section is the result of elementary computations using distance formulas from the Poincar\'e half plane model.

\begin{rem}\label{rem:halfplane}
Let $p=(a,t)$, $a,t > 0$ be a point in the Poincar\'e half plane model, and let $u \in T_p \HH^2$ denote the tangent unit vector of the geodesic line through $p$ and perpendicular to the $y$-axis, pointing towards the axis. Furthermore, let $C$ be the hyperbolic line represented by the circle centered at the origin $o$ and Euclidean radius $r$, and let $v\in T_p \HH^2$ denote the tangent unit vector of the geodesic line through $p$ and perpendicular to $C$, pointing towards $C$. Assume that $r < \sqrt{a^2+t^2}$. Then the hyperbolic distance of $p$ from the $y$-axis and from $C$ are $\arsh \frac{a}{t}$ and $\arsh \frac{t^2+a^2-r^2}{2rt}$, respectively.
In addition, the $y$-coordinates of $u$ and $v$ are $\frac{a}{\sqrt{a^2+t^2}}$ and $-\frac{t^2+r^2-a^2}{\sqrt{(r^2+a^2+t^2)^2-4r^2a^2}}$, respectively.
\end{rem}

\section{Proof of Theorem~\ref{thm:main} for $\cc(\cdot)$ and $\IC(\cdot)$}\label{sec:spheres}

First, we prove the statement for $\cc(\cdot)$. During the proof, we set $D^v = \HH^3 \setminus \left( \bigcup_{i=1}^n \bar{V}_i \right)$.
Observe that a ball $B$ is the smallest ball containing $P$ if, and only if it contains $P$, and its center belongs to the convex hull of the vertices of $P$ lying on the boundary of the ball.

Let $I$ be the set of indices such that $\frac{1}{\cos \alpha_i} = \max \left\{ \frac{1}{\cos \alpha_j} : j=1,2,\ldots,n  \right\}$.
Thus, by Remark~\ref{rem:fv}, $o=\cc(P)$ if and only if $o \in \conv \left\{ \frac{1}{\cos \beta_i} v_i : i \in I \right\}$, which is equivalent to $o \in \conv \{ v_i : i \in I \}$. Furthermore, $I$ is the set of indices with the property that $d_i^v(o) = \min \{ d_j^v(o) : j=1,2,\ldots,n \}$.
We may extend this definition for any $p \in D^v$, and let $I(p)$ denote the set of indices with the property that $d_i^v(p) = \min \{ d_j^v(p) : j=1,2,\ldots,n \}$.
Since M\"obius transformations act transitively on $\HH^3$, we need only to show the existence of a point $p \in D^v$ such that $o_p \in \conv \{ v_i(p) \subset T_p \HH^3 : i \in I(p) \}$.

For any plane $V_i$ and $\tau > 0$, consider the set $V_i(\tau)$ of points in $D^v$ at distance at most $\tau$ from $V_i$.
This set is bounded by $V_i$ and a hypersphere, which, in the model, are represented by the intersections of two spheres with the interior of $\S^2$, and share the same ideal points. Hence, if $\tau$ is sufficiently small, then the sets $V_i(\tau)$ and $V_j(\tau)$, where $i \neq j$, intersect if, and only if the $i$th and the $j$th vertices of $P$ are connected by an edge. On the other hand, if $\tau$ is sufficiently large, then all $V_i(\tau)$s intersect.
Let $\tau_0$ be the smallest value such that some $V_i(\tau_0)$ and $V_j(\tau_0)$ intersect, where $i \neq j$ and the $i$th and $j$th vertices are not neighbors, and
let $p \in V_i(\tau_0) \cap V_j(\tau_0)$.
Note that $v_i(p)$ is an inner surface normal of the boundary of $V_i(\tau_0)$ at $p$. Thus, the definition of $\tau_0$ yields that the system of inequalities $\langle x , v_i(p) \rangle > 0$, $i \in I(p)$ has no solution for $x$, from which it follows that there is no plane in $T_p \HH^3$ that strictly separates $o_p$ from the $v_i(p)$s, implying that $o_p \in \conv \{ v_i(p) : i \in I(p) \}$.
This proves the statement for $\cc(\cdot)$.
To prove it for $\IC(\cdot)$, we may apply the same argument for the face circles of $P$.

\section{Proof of Lemma~\ref{lem:integralcurves}}\label{sec:integralcurves}

The main goal of this section is to prove Lemma~\ref{lem:integralcurves}. In its formulation and proof we use the notations introduced in Section~\ref{sec:prelim}.
We note that two hyperbolic planes $V_i$ and $F_j$ intersect if, and only if the $i$th vertex of $P$ lies on the $j$th face of $P$. In this case the two planes have a common ideal point, coinciding with a tangency point of $P$. This point is the ideal point is one pair of $V_i$s and one pair of $F_j$s, and these two pairs are orthogonal.

If $q$ is a boundary point of $D$ in the Euclidean topology, by a \emph{neighborhood} of $q$ we mean the intersection of a neighborhood of $q$ with $D$, induced by the Euclidean topology of $\Re^3$. Before stating our main lemma, we note that if $h : D \to TD$ is a smooth vector field, then by the Picard-Lindel\"of Theorem for any $p \in D$ with $h(p) \neq o$ there is a unique integral curve of $h$ passing through $p$. These integral curves are either closed, or start and end at boundary points of $D$ or at points $q$ with $h(q) = o$.

\begin{lem}\label{lem:integralcurves}
Let 
\[
h : D \to TD, \,  h(p) = \sum_{i=1}^n w_i v_i(p) + \sum_{j=1}^m W_j f_j(p),
\]
be a vector field where the coefficients $w_i$ and $W_j$ are positive smooth functions of $n+m$ variables, depending on $d^v_i(p)$, $i=1,2,\ldots,n$ and $d^f_j(p)$, $j=1,2,\ldots,m$.
Assume that for any boundary point $q$ of $D$,
\begin{itemize}
\item[(i)] $q$ has a neighborhood disjoint from any closed integral curve of $h$.
\item[(ii)] If $q \in G_j$ for some value of $j$, then there is no integral curve of $h$ ending at $j$.
\item[(iii)] If $q \in F_i$ for some value of $i$ and $q \notin G_j$ for all values of $j$, then $q$ has a neighborhood in which the integral curve through any point ends at a point of $F_i$.
\item[(iv)] If $q \in \S^2$ is a tangency point of $P$, then there is a codimension 1 foliation of a neighborhood of $q$ in $D$ such that $q$ is not an ideal point of any leaf, and for any point $p$ on any leaf $h(p) \neq o$, the integral curve through $p$ crosses the leaf, either in the direction of $q$ or from this direction, independently of the choice of $p$, the leaf and $q$.
\end{itemize}
Then $h(p) = o_p$ for some $p \in D$. 
\end{lem}

First, we prove Lemma~\ref{lem:topological}, where, by $\BB^d$, we mean the closed $d$-dimensional Euclidean unit ball centered at $o$.

\begin{lem}\label{lem:topological}
Let $X = (\inter \BB^{d+1}) \setminus (1-\varepsilon) \BB^{d+1}$, where $0 < \varepsilon < 1$, and $d \geq 2$.
Let $Z_1, \ldots, Z_k$ be pairwise disjoint closed sets in $X$, where $k \geq 1$.
If $X \setminus Z_i$ is connected for all values of $i$ then $X \setminus \left( \bigcap_{i=1}^k Z_i \right)$ is connected.
\end{lem}

\begin{proof}[Proof of Lemma~\ref{lem:topological}]
We prove the assertion by induction for $k$.
If $k=1$, then the statement is trivial. Assume that Lemma~\ref{lem:topological} holds for any $k-1$ closed sets.
Let $Z'= \bigcup_{i=1}^{k-1} Z_i$. Then $\overline{Z}_k = X \setminus Z_k$ and $\overline{Z'} = Z \setminus Z'$ are open sets whose union is $X$. Consider the Mayer-Vietoris exact sequence \cite{Massey} of these subspaces:
\[
H_1(X) \to H_0(\overline{Z_k} \cap \overline{Z'}) \to H_0(\overline{Z_k} \oplus \overline{Z'}) \to H_0(X) \to 0 .
\]
Note that by the induction hypothesis, $\overline{Z'}$ is connected.
On the other hand, since $\S^d$ is a deformation retract of $X$, their homology groups coincide, implying that
$\rank H_1(X) = 0$, $\rank H_0(X) = 1$.
Since $X$ is locally path-connected, any connected subset of $X$ is path-connected, and thus, $\rank H_0(X)$ is the number of connected components of $X$, implying that $\rank(H_0(\overline{Z_k} \oplus \overline{Z'})) = 2$, and $\rank(H_0(\overline{Z_k} \cap \overline{Z'}))=t$, where $t$ is the number of the connected components of $\overline{Z_k} \cap \overline{Z'}$.
The exactness of the Mayer-Vietoris sequence yields that $1-2+t=0$, that is, $t=1$.
\end{proof}

\begin{proof}[Proof of Lemma~\ref{lem:integralcurves}]
We prove Lemma~\ref{lem:integralcurves} by contradiction.
Assume that $h(p) \neq o$ for any $p \in D$, and let $S$ denote the set of tangency points of $P$.
Furthermore, let $Z$ denote the set of the points of $D$ belonging to a closed integral curve. For $i=1,2,\ldots,n$, let $Y_i$ denote the set of points whose integral curve terminates at a point of $H_i$, and let $W_s$ be the set of the points with their integral curves ending at $s \in S$. By (iii), every set $Y_i$ is open, and it is easy to see that every set $W_s$ is closed.

First, assume that for any $s \in S$, the integral curve through any point $p$ on a leaf of the codimension 1 foliation in a neighborhood of $s$ points away from the direction of $s$. This implies, in particular, that $W_s = \emptyset$ for all $s \in S$.
For all $q \in \bd D$, let $V_q$ denote a neighborhood of $q$ satisfying the conditions of the lemma. By the definition of induced topology, $V_q = V_q^*$ for some neighborhood of $q$ in $\Re^3$. We may assume that $V_q^*$ is open for all $q \in \bd D$.
Since the sets $V_q^*$ cover the compact set $\bd D$, we may choose a finite subfamily that covers $\bd D$. By finiteness, it follows that there is some $\varepsilon > 0$ such that the set $D_{\varepsilon}$ of points at Euclidean distance less than $\varepsilon$ from $\bd D$ is disjoint from $Z$.
On the other hand, $D_{\varepsilon}$ is connected, yet it is the disjoint union of the finitely many open sets $Y_i \cap D_{\varepsilon}$, a contradiction.

Assume now that that for any $s \in S$, the integral curve through any point $p$ on a leaf of the codimension 1 foliation in a neighborhood $V_s$ of $s$ points towards $s$. By this, if $s \in S$ is the tangency point connecting the $i$th and $j$th vertices, then $V_s \subseteq W_s \cup Y_i \cup Y_j$.
On the other hand, by (iii), all $Y_i$s are connected. Thus, for any walk on the edge graph of $P$ starting at the $k$th and ending at the $l$th vertex, there is a continuous curve in $D$ starting at a point of $Y_k$ and ending at a point of $Y_l$, and passing through points of only those $Y_i$s and $W_s$s for which the associated vertices and edges of $P$ are involved in the walk. In addition, the curve may pass arbitrarily close to $\bd D$, measured in Euclidean metric.

We choose the set $D_{\varepsilon}$ as in the previous case. Note that $D_{\varepsilon}$ is homeomorphic to $(\inter \BB^3) \setminus (1-\varepsilon) \BB^3$, and thus, we may apply Lemma~\ref{lem:topological} with the $W_s$s playing the roles of the $Z_j$s. Then it follows that for some $s \in S$, $D_{\varepsilon} \setminus W_s$ is disconnected. Since the union of finitely many closed sets is closed, there are some $Y_k$ and $Y_l$ in different components.
By Steinitz's theorem \cite{Steinitz1, Steinitz2}, there is a path in the edge graph of $P$ that connects the $k$th and $l$th vertices and avoids the edge associated to $s$. Hence, there is a continuous curve in $D$, starting at a point of $Y_k$ and ending at a point of $Y_l$ that avoids $W_s$; a contradiction.
\end{proof}

\section{Proof of Theorem~\ref{thm:main}}\label{sec:main}

\subsection{Barycenter of the vertices: $\cm_0(P)$}

We show that Theorem~\ref{thm:main} for the barycenter of its vertices is an immediate consequence of Theorem~\ref{thm:circles}.

By Remark~\ref{rem:fv}, we have $\cm_0(P) = \frac{1}{n} \sum_{i=1}^n \frac{1}{\cos \alpha_i} v_i$.
Thus, it is sufficient to show that the conditions of Theorem~\ref{thm:circles} are satisfied for the family of vertex circles of $P$
with the weight functions $w_i(t) = \frac{1}{\cos t}$ for all $i$s.

First, observe that if $n=4$ (i.e. if $P$ is a tetrahedron), then $\cm_0(P) = o$ if $P$ is regular.
Thus, we may assume that $n \geq 5$. Note that the weight functions $w_i(t) = \frac{1}{\cos t}$ are positive smooth functions on $\left( 0, \frac{\pi}{2} \right)$ and satisfy $\lim_{t \to \frac{\pi}{2}} w_i(t) = \infty$. Furthermore, since
$| I(q)|  \leq 2$ for all points $q \in \S^2$, the inequality in (\ref{eq:circle_condition}) holds, and Theorem~\ref{thm:circles} implies Theorem~\ref{thm:main} for $\cm_0(\cdot)$.

\subsection{Center of mass of the wire model: $\cm_1(P)$}\label{subsec:cm1}

Let $E$ denote the set of edges of the edge graph of $P$; that is, $\{i,j\} \in E$ if, and only if the $i$th and $j$th vertices are connected by an edge.
An elementary computation yields that if $\{i,j\} \in E$, the length of the corresponding edge of $P$ is $\tan \alpha_i + \tan \alpha_j$, and its center of mass is $\frac{1}{2} \left( \frac{v_i}{\cos \alpha_i} + \frac{v_j}{\cos \alpha_j} \right)$.
Thus, letting $A= \sum_{\{i,j\} \in E} \left( \tan \alpha_i + \tan \alpha_j \right)$, we have
\begin{equation}\label{eq:cm1}
\cm_1(P) = \frac{1}{2A} \sum_{\{i,j\} in E}  \left( \tan \alpha_i + \tan \alpha_j \right) \left( \frac{v_i}{ \cos \alpha_i} + \frac{v_j}{\cos \alpha_j} \right) .
\end{equation}

Set $D^v = \HH^3 \setminus \left( \bigcup_{i=1}^n \bar{V}_i \right)$, and define the function $h^v : D^v \to T D^v$ as
\begin{equation}\label{eq:cm1_hyp}
h^v(p)= \sum_{\{i,j\} \in E}  \left( \frac{1}{\sinh d^v_i(p)} + \frac{1}{\sinh d^v_j(p)} \right) \left( \coth (d_i(p)) v_i(p) + \coth(d_j(p))  v_j(p) \right).
\end{equation}
Then $h^v$ is a smooth function on $D^v$ and the coefficient of each vector $v_i(p)$ is positive.
By Remark~\ref{rem:parallelism}, it follows that if there is a point $p \in D^v$ such that $h^v(p) = o$, then, choosing a M\"obius transformation $T$ that maps $p$ into $o$, we have $\cm_1(T(P)) = o$.
We denote the restriction of $h^v$ to $D$ by $h$, and show that $h$ satisfies the conditions in Lemma~\ref{lem:integralcurves}. 

Let $q$ be a boundary point of $D$ in some plane $F_j$ associated to a face circle of $P$. Assume that $q$ is not contained in $V_i$ for any value of $i$.
Observe that if the $i$th vertex lies on the $j$th face, then $v_i(q)$ and $f_j(q)$ are orthogonal, and otherwise $v_i(q)$ points inward to $D$.
Thus, by the continuity of $h^v$, there is no integral curve of $h$ that ends at $q$, and $q$ has a neighborhood disjoint from the set $Z$ of the points of the closed integral curves of $h$. If $q$ is contained in $V_i$ for some $i$, then a slight modification of this argument can be applied. This proves (ii) in Lemma~\ref{lem:integralcurves}.

Let $q$ be a point of some $V_i$ not contained in any of the $F_j$s. Then, denoting the coefficient of $v_j(p)$ by $\mu_j(p)$ for any $j$, we have that $\frac{\mu_i(p)}{\mu_j(p)} \to \infty$ for all $j \neq i$, as $p \to \infty$, which shows that if $p$ is `close' to $q$, then $h(p)$ is `almost orthogonal' to $F_j$. This shows (iii), and the fact that a neighborhood of $q$ is disjoint from $Z$.

Finally, let $q$ be a tangency point of $P$. Without loss of generality, we may assume that $q$ is the ideal point of $V_1$, $V_2$, $F_1$ and $F_2$.
To prove (iv), we imagine the configuration in the Poincar\'e half space model, with $q$ as the `point at infinity';
geometrically, it means that we apply an inversion to $\Re^3$ about a sphere centered at $q$.
Then $D$ is contained in the half-infinite cylinder bounded by the four vertical planes $V_1, V_2, F_1$ and $F_2$ (for the definition of vertical and spherical plane, see Subsection~\ref{subsec:idea}). Note that the cross section of this cyclinder is a rectangle, and that all other $V_i$s and $F_j$s are spherical planes centered at ideal points of $D$ in the Euclidean plane $\{ z=0 \}$.

For any $t > 0$, let $D_t$ denote the intersection of the set $\{ z = t \}$ with $D$. We remark that $\{ z = t \}$ is a horosphere whose only ideal point is $q$, and thus, the sets $D_t$, where $t$ is sufficiently large, form a codimension $1$ foliation of a neighborhood of $q$ in $D$.
Hence, to show that the conditions of Lemma~\ref{lem:integralcurves} are satisfied, it is sufficient to show that if $t$ is sufficiently large, then $h(p)$ has a positive $z$-coordinate for any $p \in D_t$.

For any $\{ i,j\} \in E$, denote the term in $h(p)$ belonging to $\{ i,j\}$ by $h_{i,j}(p)$, and the $z$-coordinate of $h_{i,j}(p)$
by $z_{i,j}(p)$.
Let $\{ i,j\}$ and $\{1,2\}$ be disjoint. Note that the closure of $D_t$ is compact. Thus, by Remark~\ref{rem:halfplane}, if $t \to \infty$, then $h(p)$ uniformly converges to $0$.
Assume that $\{i,j\} \cap \{ 1,2\}$ is a singleton, say $i=1$ and $j \neq 2$. Then, by Remark~\ref{rem:halfplane}, the $z$-coordinate of $\coth d^v_1(p) v_1(p)$ is $1$, and that of $\coth d^v_j(p) v_j(p)$ is less than $1$. Thus, $z_{1,j}(p) > 0$ in this case.
Finally, $z_{1,2}(p) > C t$ for any $p \in D_t$ for some universal constant $C> 0$.
Thus $h(p)$ has a positive $z$-coordinate for large values of $t$, and Lemma~\ref{lem:integralcurves} implies Theorem~\ref{thm:main} for the case of $\cm_1(\cdot)$.

\subsection{Center of mass of the paper model: $\cm_2(P)$}\label{subsec:cm2}

Let $I$ denote the edge set of the vertex-face incidence graph of $P$; that is, $(i,j) \in I$ if, and only if the $i$th vertex lies on the $j$th face.
Consider some $(i,j) \in I$. Then there are exactly two edges of $P$ adjacent to both the vertex and the face.
Let the tangency points on these two edges be denoted by $e^1_{i,j}$ and $e^{2}_{i,j}$. Then, by Remark~\ref{rem:fv}, the points $\frac{v_i}{\cos \alpha_i}$, $e^{1}_{i,j}$, $\cos \beta_j f_j$ and $e^2_{i,j}$ are coplanar, and they are the vertices of a symmetric right trapezoid $Q_{i,j}$ (cf. Figure~\ref{fig:paper}).
Note that $\bd P$ can be decomposed into the mutually nonoverlapping trapezoids $Q_{i,j}$, $(i,j) \in I$.
An elementary computation yields that the center of gravity of $Q_{i,j}$ is
\[
\frac{1}{3} \left( \frac{2 \tan^2 \alpha_i + \sin^2 \beta_j}{\tan^2 \alpha_i + \sin^2 \beta_j} \cos \beta_j f_j + \frac{\tan^2 \alpha_i + 2\sin^2 \beta_j}{\tan^2 \alpha_i + \sin^2 \beta_j} \frac{1}{\cos \alpha_i} v_i \right).
\]
The area of $Q_{i,j}$ is $\tan \alpha_i \sin \beta_j$. Thus, letting $A= \sum_{ (i,j) \in I} \tan \alpha_i \sin \beta_j$,
we have
\begin{equation}\label{eq:cm2}
\cm_2(P) = \frac{1}{3A} \sum_{ (i,j ) \in I} \tan \alpha_i \sin \beta_j \left( \frac{2 \tan^2 \alpha_i + \sin^2 \beta_j}{\tan^2 \alpha_i + \sin^2 \beta_j} \cos \beta_j f_j + \frac{\tan^2 \alpha_i + 2\sin^2 \beta_j}{\tan^2 \alpha_i + \sin^2 \beta_j} \frac{1}{\cos \alpha_i} v_i \right) .
\end{equation}

\begin{figure}[ht]
\begin{center}
\includegraphics[width=0.6\textwidth]{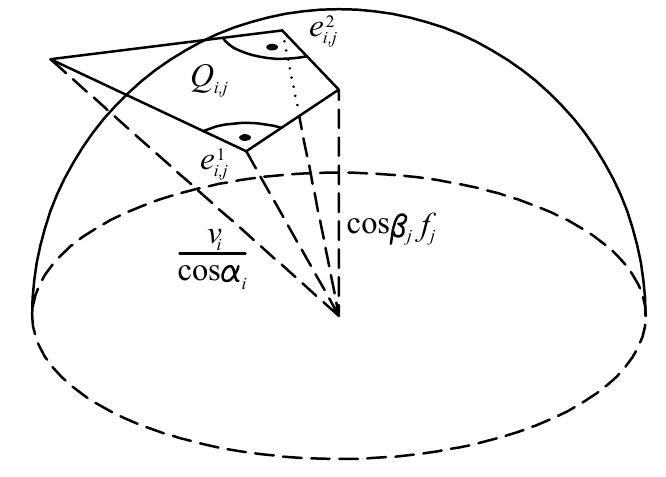}
\caption{The right trapezoid $Q_{i,j}$}
\label{fig:paper}
\end{center}
\end{figure}

Let us define the smooth vector field $h:D \to TD$ as
\begin{equation}\label{eq:function_faces}
h(p) = \sum_{(i,j) \in I} h_{i,j}(p),
\end{equation}
where
\[
h_{i,j}(p) = \frac{1}{\sinh d^v_i \cosh d^f_j} \left( \frac{2 \cosh^2 d^f_j + \sinh^2 d^v_i}{\cosh^2 d^f_j + \sinh^2 d^v_i} \tanh d^f_j f_j(p) + \frac{\cosh^2 d^f_j + 2 \sinh^2 d^v_i}{\cosh^2 d^f_j + \sinh^2 d^v_i} \coth d^v_i  v_i(p) \right).
\]
Here, for simplicity, we set $d^v_i = d^v_i(p)$ and $d^f_j=d^f_j(p)$.
The function $h$ is a smooth function on $D$ with positive coefficients. Furthermore, by Remark~\ref{rem:parallelism}, if $h(p) = o$ for some $p \in D$ and $T$ is a M\"obius transformation mapping $p$ into $o$, then $\cm_2(T(P))=o$. Similarly like in Subsection~\ref{subsec:cm1}, we show that the conditions of Lemma~\ref{lem:integralcurves} are satisfied for $h$.

To prove (ii) and (iii) we apply the same argument as in Subsection~\ref{subsec:cm1}.
To prove (iv), we follow the line of the same proof, and imagine the configuration in the half space model.
Let $q$ be the ideal point of $V_1, V_2, F_1$ and $F_2$. Then $D$ is bounded by the vertical planes $V_1, V_2, F_1$ and $F_2$ which form a rectangle based half-infinite cylinder. We adapt the notations from the previous subsection, and set $D_t = D \cap \{ z=t \}$ for all $ > 0$.
We denote the $z$-coordinate of $h_{i,j}(p)$ by $z_{i,j}(p)$, and show that their sum is positive if $t$ is sufficiently large.

By Remark~\ref{rem:halfplane} and an elementary computation, if $i \notin \{1,2\}$, then $z_{i,j}(p)$ uniformly tends to zero for all $p \in D_t$ as $t \to \infty$.
To examine the remaining cases, for $i=1,2$, let $x_i(p)$ denote the Euclidean distance of the point $p$ from $V_i$. Then $x_1(p) + x_2(p) = x$ is the Euclidean distance of $V_1$ and $V_2$. 
By Remark~\ref{rem:halfplane}, there is some constant $C_1 > 0$ independent of $p$, $t$, $i$ and $j$ such that for all $p \in D_t$, $j \geq 3$ and $i \in \{ 1,2 \}$,
we have $z_{i,j}(p) \geq - \frac{C_1}{x_i}$. Similarly, there is some constant $C_2 > 0$ independent of $p,t,i,j$ such that for all $p \in D_t$, $i,j \in \{1,2\}$,
we have $z_{i,j}(p) \geq \frac{C_2 t^2}{x_i}$. This implies that if $t$ is sufficiently large (and in particular, if $t > \sqrt{\frac{C_1 k}{C_2}}$, where $k$ is the maximal degree of a vertex of $P$), then the $z$-coordinate of $h(p)$ is positive for all $p \in D_t$.
From this, Theorem~\ref{thm:main} readily follows for $\cm_2(\cdot)$.


\subsection{Circumcenter of mass: $\ccm(\cdot)$}\label{subsec:ccm}

In this subsection we assume that $P$ is simplicial.

Similarly like in Subsection~\ref{subsec:cm2}, we denote by $I$ the set of edges of the vertex-face incidence graph of $P$, and
by $\V_j = \{ a_j, b_j, c_j\}$ the set of the indices of the vertices adjacent to the $j$th face of $P$.

Let the convex hull of the $j$h face of $P$ and $o$ be denoted by $S_j$.
To compute $\ccm(P)$, we need to compute the volume and the circumcenter of $S_j$, which we denote by $m_j$ and $p_j$.
To do this, in the next lemma for simplicity we omit the index $j$, and in addition denote $\tan \alpha_{x_j}$ by $t_x$ for $x \in \{a,b,c\}$.



\begin{lem}\label{lem:ccm}
The volume of $S_j$ is
\begin{equation}\label{eq:simplex_volume}
m_j = \frac{1}{3}\sqrt{t_a t_b t_c \left( t_a+t_b+t_c-t_a t_b t_c \right) } .
\end{equation}
The circumcenter of $S_j$ is
\begin{equation}\label{eq:cc1}
p_j = \sum_{s \in \{a,b,c\}} N_s v_s,
\end{equation}
where
\begin{equation}\label{eq:cc2}
N_a = \frac{(t_b+t_c)\left( (t_b+t_c) t_a^2 +(2t_b^2 t_c^2 + t_b^2 + t_c^2) t_a - t_b t_c (t_b +t_c))\right)}{4 t_at_bt_c \left( t_a+t_b+t_c-t_a t_b t_c \right)} ,
\end{equation}
and $N_b$ and $N_c$ are defined analogously.
\end{lem}

\begin{proof}
Note that the three edges of $S_j$ starting at $o$ are of length $\frac{1}{\cos \alpha_x}$ with $x \in \{a,b,c\}$.
Furthermore, the edge opposite of the one with length $\frac{1}{\cos \alpha_{x}}$ is $t_y+t_z$, where $\{x,y,z\} = \{a,b,c\}$.
Thus, the volume of $S_j$ can be computed from its edge lengths using a Cayley-Menger determinant.
It is worth noting that since the projection of $F$ onto $\S^2$ is a spherical triangle of edge lengths $\alpha_{a}+\alpha_{b}$, $\alpha_{a}+\alpha_c$ and $\alpha_{b}+\alpha_{c}$, and such a triangle is spherically convex, its perimeter is $\alpha_a+\alpha_b+\alpha_c < \pi$. From this an elementary computation yields that $t_i+t_j+t_k  - t_i t_j t_k > 0$, and the formula in (\ref{eq:simplex_volume}) is valid.
                                                                                                                   
We compute $p_j$. Since the vectors $v_a, v_b$ and $v_c$ are linearly independent, we may write this point in the form $p_j = \sum_{s \in \{ i,j,k\}} N_s v_s$ for some coefficients $N_a, N_b, N_c$.
We multiply both sides of this equation by $v_r$ with some $r \in \{a,b,c \}$.                            
Since all $v_i$s are unit vectors, we have that $\langle v_s, v_r\rangle = \cos (\alpha_s + \alpha_r)$ if $s \neq r$, and $\langle u_r,u_r \rangle = 1$.
On the other hand, for any value of $r$, $p_j$ is contained in the plane with normal vector $v_r$ passing through the point $\frac{u_r}{2 \cos \alpha_r}$.
Hence, it follows that $[N_a,N_b,N_c]^T$ is the solution of the system of linear equations with coefficient matrix
\[
\left[
\begin{array}{ccc}
1 & \cos(\alpha_a + \alpha_b) & \cos(\alpha_a+\alpha_c)\\
\cos(\alpha_a+\alpha_b) & 1 & \cos(\alpha_b+\alpha_c)\\
\cos(\alpha_a+\alpha_c) & \cos(\alpha_b+\alpha_c) & 1
\end{array}
\right]
\]
and with constants $\frac{1}{2 \cos \alpha_r}$, where $r=a,b,c$.
The determinant of the coefficient matrix is $36 (m_j)^2 (1+t_a^2)(1+t_b^2)(1+t_c^2) > 0$. Thus, this system has a unique solution, which can be computed by Cramer's rule, yielding the formula in (\ref{eq:cc2}).
\end{proof}

For $s=1,2,\ldots,n$, let us denote the value $\csch d^v_s(p) = \frac{1}{\sh d^v_s(p)}$ by $\tau_s(p)$.
Observe that Remark~\ref{rem:parallelism} implies that $\csch d^v_s(o) = \tan \alpha_s$.
For any $p \in D$, let us define the vector field
\begin{equation}\label{eq:ccm_function}
h(p) = \sum_{j=1}^m \sum_{s \in \V_j} B_s(p) v_s(p),
\end{equation}
where, using the notation $\V_j = \{ a,b,c \}$ and for brevity omitting the variable $p$, we have
\begin{equation}\label{eq:ccm_coeff}
B_{a}(p) = \frac{\tanh d_{a} \left( \tau_{b} + \tau_{c}\right) \left( \tau_a^2 \left( \tau_b + \tau_c \right) +
\tau_a \left(2 \tau_b^2 \tau_c^2 + \tau_b^2 + \tau_c^2 \right) - \tau_b \tau_c \left( \tau_b + \tau_c \right) \right)}{
\sqrt{\tau_a \tau_b \tau_c \left( \tau_a + \tau_b + \tau_c - \tau_a \tau_b \tau_c\right)}} .
\end{equation}
If $h(p) = o$ and $T$ is a M\"obius transformation that maps $p$ into $o$, then $\ccm(T(P)) = o$.
Thus, to prove the statement it is sufficient to prove that for some $p \in D$, $f(p) = o$.
To do this, we check that the conditions of Lemma~\ref{lem:integralcurves} are satisfied.

Let $Z$ denote the set of points of $D$ whose integral curve is closed.
Since for any value of $j$, $F_j$ is perpendicular to any $V_i$ with $(i,j) \in I$ and does not intersect any other $V_i$,
similarly like in Subsection~\ref{subsec:cm1}, it follows that if $q \in F_j$ for some plane $F_j$ associated to a face circle of $P$, then $q$
has a neighborhood disjoint from $Z$, and no integral curve ends at $q$.

Let $q \in V_i$ for some value of $i$.
It is an elementary computation to check that if $\alpha+\beta+\gamma = \pi$, and $0 < \alpha, \beta, \gamma < \frac{\pi}{2}$, then
$\tan \alpha + \tan \beta + \tan \gamma = \tan \alpha \tan \beta \tan \gamma$. This and Remark~\ref{rem:parallelism} implies
that if $p \to q$ and $i \in \{a,b,c \}$, then the denominator of $B_a(p)$ tends to zero.
Since the numerator tends to a positive number if $a=i$, and to zero if $i=b$ or $i=c$, it follows that
if $i \in \V_j$, then the length of $\sum_{s \in \V_j} B_s(p) v_s(p)$ tends to $\infty$, and its direction tends to that of $v_i(p)$.
Since $i \notin \V_j$ implies that $\sum_{s \in \V_j} B_s(p) v_s(p)$ can be continuously extended to $q$, it follows that
the angle of $h(p)$ and the external normal vector of $V_i$ at $q$ is `almost' zero in a suitable neighborhood of $q$.
This yields (iii).

We prove (iv) in a Poincar\'e half space model with $q$ being the `point at infinity'.
Without loss of generality, we may assume that $q$ is the ideal point of $V_1$, $V_2$, $F_1$ and $F_2$. Then these two pairs of hyperbolic planes are represented by two perpendicular pairs of vertical hyperbolic planes.
As before, let $D_t$ denote the set of points in $D$ with $z$-coordinates equal to $t$. We show that the $z$-coordinate of $h(p)$ is positive for any $p \in D_t$, if $t$ is sufficiently large. For any $j$ and any $i \in \V_j$, let us denote the $z$-coordinate of $B_i(p) v_i(p)$ by $z^j_i(p)$.

Let $p \in D_t$, and denote by $x_1$ and $x_2$ the Euclidean distance of $p$ from $V_1$ and $V_2$, respectively.
Consider some value of $j$. If $\V_j$ is disjoint from $\{1,2\}$, then Remark~\ref{rem:halfplane} and (\ref{eq:ccm_coeff}) shows that there is some $C_1 > 0$
independent of $p$ such that $|z^j_i(p)| \leq \frac{C_1}{t^2}$ if $t$ is sufficiently large.
Assume that $\V_j$ contains exactly one of $1,2$, say $1$.
Then, an elementary computation and Remark~\ref{rem:halfplane} yields the existence of some $C_2, C_3 > 0$ independent of $p$ such that $|z^F_1(p)| \leq \frac{C_2}{t^2}$, and for $1 \neq i \in \V_j$, $|z^j_i(p)| \leq  C_3 \frac{t^2}{x_1^2}$.

Finally, let $\V_j=\{1,2,i\}$. Note that in this case $j=1$ or $j=2$. Furthermore, since $P$ is simplicial, we have that the Euclidean radius of the hemisphere representing $V_i$ is $\frac{x_1+x_2}{2}$, and the Euclidean distance of the center of this hemisphere from the projection of $p$ onto the $\{z=0\}$ plane is
$\sqrt{\left( \frac{x_1-x_2}{2} \right)^2 + y_j^2}$, where $y_j$ is the Euclidean distance of $p$ from $F_j$ (cf. Figure~\ref{fig:ccm}).
An elementary computation yields that by this and Remark~\ref{rem:halfplane}, the denominator in (\ref{eq:ccm_coeff}) is $\frac{t^3 (x_1+x_2)y_j}{(t^2+y_j^2-x_1x_2)x_1x_2}$. Using this, we have $|z_1^j(p) | \leq \frac{2x_1^2}{x_2 y_j} t$, $|z_2^j(p) | \leq \frac{2x_2^2}{x_1 y_j} t$ and
$z_i^j(p) \geq \frac{x_1+x_2}{2x_1^2 x_2^2 y_j} t^3$ if $t$ is sufficiently large.
Using these estimates, we have $z_1^j(p) + z_2^j(p) + z_i^j(p) \geq \frac{C_4 t^3}{x_1^2 x_2^2}$ for some $C_4 > 0$ independently of $t$ and $p$.
Thus, there is some $C > 0$ such that if $t$ is sufficiently large, $\sum_{j=1}^n \sum_{i \in \V_j} z_i^j(p) \geq C t^3$, and, in particular, this expression is positive. The regions $D_t$ form a codimension 1 foliation of a neighborhood of $q$, and thus Theorem~\ref{thm:main} follows from Lemma~\ref{lem:integralcurves}.

\begin{figure}[ht]
\begin{center}
\includegraphics[width=0.45\textwidth]{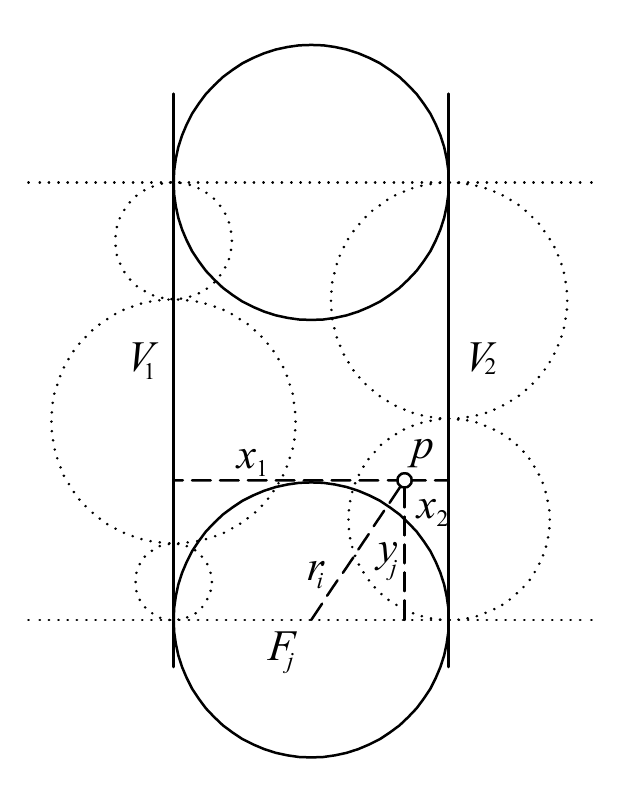}
\caption{The ideal points of hyperbolic planes associated to a simplicial polyhedron in the Euclidean plane $\{ z=0\}$. Continuous lines represent planes associated to vertex circles. Dotted lines represent planes associated to face circles.}
\label{fig:ccm}
\end{center}
\end{figure}

\subsection{Points of the Euler line}\label{subsec:euler}

Again, we assume that $P$ is simplicial.
Using the calculations in Subsection~\ref{subsec:cm2}, we have that the center of mass of $P$ is
\[
\cm_3(P) = \frac{1}{4A} \sum_{ (i,j ) \in I} \tan \alpha_i \sin \beta_j \cos \beta_j \left( \frac{2 \tan^2 \alpha_i + \sin^2 \beta_j}{\tan^2 \alpha_i + \sin^2 \beta_j} \cos \beta_j f_j + \frac{\tan^2 \alpha_i + 2\sin^2 \beta_j}{\tan^2 \alpha_i + \sin^2 \beta_j} \frac{1}{\cos \alpha_i} v_i \right),
\]
where $A=\sum_{ (i,j ) \in I} \tan \alpha_i \sin \beta_j \cos \beta_j$.

By Remark~\ref{rem:parallelism}, we define the smooth vector field $h_{cm} : D \to TD$ as
\begin{equation}\label{eq:cm}
h_{cm} (p) = \sum_{(i,j) \in I} h_{i,j}(p),
\end{equation}
where
\[
h_{i,j}(p) = \frac{\sinh d^f_j}{\sinh d^v_i \cosh^2 d^f_j} \left( \frac{2 \cosh^2 d^f_j + \sinh^2 d^v_i}{\cosh^2 d^f_j + \sinh^2 d^v_i} \tanh d^f_j f_j(p) + \frac{\cosh^2 d^f_j + 2 \sinh^2 d^v_i}{\cosh^2 d^f_j + \sinh^2 d^v_i} \coth d^v_i  v_i(p) \right).
\]
Furthermore, for any $\lambda \in (0,1)$, we set
$h_{\lambda}(p) = \lambda h_{cm}(p) + (1-\lambda) h_{ccm}(p)$,
where $h_{ccm} : D \to TD$ is the vector field defined in (\ref{eq:ccm_function}).
We observe that if there is some $p \in D$ such that $h_{\lambda}(p) = o$, and $T$ is a M\"obius transformation moving $p$ to $o$, then $o = \lambda \cm_3(T(P))+(1-\lambda)\ccm(T(P))$.

We show that the conditions of Lemma~\ref{lem:integralcurves} are satisfied for $h_{\lambda}$.
Note that since $\lambda \in (0,1)$, all coefficients in the definition of $h_{\lambda}$ are positive.
To check (i), (ii) and (iii), we may apply an argument similarly like before.
To prove (iv), again we represent the configuration in the half space model.
Let $D_t$ be the intersection of $D$ with the horosphere $\{ z=t \}$, and $z_{cm}(p)$ and $z_{\lambda}(p)$ denote the $z$-coordinate of $h_{cm}(p)$  and $h_{\lambda}(p)$, respectively.
Then an elementary computation yields by Remark~\ref{rem:halfplane} that there is some $\bar{C} > 0$ such that $|z_{cm}(p)| \leq \bar{C}$ for all $p \in D_t$, if $t$ is sufficiently large.
Thus, by the estimates in Subsection~\ref{subsec:ccm} and since $\lambda < 1$ it follows that if $t$ is sufficiently large, then $z_{\lambda}(p) > 0$ for all $p \in D_t$. Consequently, Lemma~\ref{lem:integralcurves} can be applied, and Theorem~\ref{thm:main} holds for the considered point of the Euler line.

\section{Proof of Theorem~\ref{thm:circles}}\label{sec:circles}

To prove Theorem~\ref{thm:circles}, we follow the line of the proof of Theorem~\ref{thm:main}. To do this, we need a lemma for polyhedral regions in Euclidean space.

\begin{lem}\label{lem:polytopes}
Let $S_1, \ldots, S_k$ be closed half spaces in $\Re^d$, with outer normal vectors $u_1,\ldots,u_k$. Then there are unit normal vectors $v_1, \ldots, v_m$ such that $\langle u_i, v_j \rangle \leq 0$, for all $1 \leq i \leq l$ and $1 \leq j \leq m$, and for arbitrary closed half spaces $S'_1, \ldots, S'_m$ with outer unit normal vectors $v_1, \ldots, v_m$, respectively, the set $Q=\left( \bigcap_{i=1}^k S_i \right) \cap \left( \bigcap_{j=1}^m S'_j \right)$ is bounded.
\end{lem}

\begin{proof}
First, observe that the property that $Q$ is bounded is equivalent to the property that there is no unit vector $v \in \S^{d-1}$ such that  $\langle v, u_i \rangle \leq 0$ and $\langle v, v_j \rangle \leq 0$ holds for all $1 \leq i \leq k$ and $1 \leq j \leq m$.
In other words, $Q$ is bounded if, and only if the open hemispheres of $\S^{d-1}$, centered at the $u_i$s and the $v_j$s, cover $\S^{d-1}$.
If $\bigcap_{i=1}^k S_i$ is bounded, there is nothing to prove, and thus, we may consider the set $Z$ of vectors in $\S^{d-1}$ not covered by any open hemisphere centered at some $u_i$. Note that since $Z$ is the intersection of finitely many closed hemispheres, it is compact. Let $G(v)$ denote the open hemisphere centered at $v$. Then the family $\{ G(v) : v \in Z\}$ is an open cover of 
$S$, and thus it has a finite subcover $\{ G(v_j) : i=1,\ldots, m \}$. By its construction, the vectors $v_1,\ldots, v_m$ satisfy the required conditions.
\end{proof}

Now we prove Theorem~\ref{thm:circles}, and for any $i=1,2,\ldots,n$, we let $\rho_i$ denote the spherical radius of $C_i$.
We imagine $\S^d$ as the set of ideal points of the Poincar\'e ball model of $\HH^{d+1}$. Then each spherical cap is associated to a closed hyperbolic half space.
We denote the half space associated to $C_i$ by $\bar{H}_i$, and the hyperplane bounding $\bar{H}_i$ by $H_i$.
Let $D = \HH^{d+1} \setminus \left( \bigcup_{i=1}^n \bar{H}_i \right)$, and note that as $\rho_i < \frac{\pi}{2}$ for all indices, $D$ is an open, convex set in $\HH^{d+1}$ containing the origin $o$.

For any $p \in D$, let us define the function $f_i(d_i) = w_i(\arccos \tanh d_i)$. Then $f_i : (0,\infty) \to (0,\infty)$ is a positive smooth function on its domain satisfying $\lim_{d \to 0+0} f_i(d) = \infty$. Let $v_i(p) \in T_p \HH^{d+1}$ denote the unit tangent vector of the geodesic half line starting at $p$ and perpendicular to $H_i$, and let $d_i(p)$ denote the hyperbolic distance of $p$ from $H_i$. Finally, let the smooth vector field $f : D \to TD$ be defined as
\[
f(p) = \sum_{i=1}^n f_i(d_i(p)) v_i(p).
\]
By (\ref{eq:angledistance}) and (\ref{eq:gradient}), if $T$ is a M\"obius transformation mapping $p$ into $o$, then $f(p) = \sum_{i=1^n} w_i(\rho_T(C_i)) c_T(C_i)$. Since hyperbolic isometries act transitively on $\HH^{d+1}$, it is sufficient to show that $f(p) = o_p$ for some $p \in D$.

We prove it by contradiction, and assume that $f(p) \neq o_p$ for any $p \in D$. Consider the integral curves of this vector field. Then, by the Picard-Lindel\"of Theorem, they are either closed, or start and terminate at boundary points of $D$.
On the other hand, since $f_i$ is smooth for all values of $i$, $f_i$ has an antiderivative function $F_i$ on its domain. It is easy to check that
$\grad (-\sum_{i=1}^n F_i(d_i(p))) = f(p)$, implying that $f$ is a gradient field, and thus it has no closed integral curves.

Our main tool is the next lemma. To state it we define a \emph{neighborhood} of a point $q$ in the boundary of $D$ as the intersection of $D$ with a neighborhood of $q$ in $\Re^{d+1}$ induced by the Euclidean topology (cf. Section~\ref{sec:integralcurves}). Recall from Theorem~\ref{thm:circles}  that if $q \in \S^d$, then $I(q)$ denotes the set of indices of the spherical caps $C_i$ that contain $q$ in their boundaries.

\begin{lem}\label{lem:intcurves_ddim}
Let $q$ be a boundary point of $D$, and if $q \notin \S^d$, then let $I(q)$ denote the set of indices such that $q \in H_i$.
\begin{itemize}
\item[(a)] If $q \notin S^d$, then $q$ has a neighborhood $V$ such that any integral curve intersecting $U$ terminates at a point of $H_j$ for some $j \in I(q)$.
\item[(b)] If $q \in \S^d$, then there is no integral curve terminating at $q$.
\end{itemize}
\end{lem}

\begin{proof}
First, we prove (a) for the case that $I(q) = \{ i \}$ is a singleton. Let $v$ be the external unit normal vector of $\bd D$ at $q$.
For any $p \in D$, if $p \to q$, then $f_i( d_i(p)) \to \infty$, and $v_i(p)$ tends to a vector of unit hyperbolic length, perpendicular to $H_i$ at $q$ and pointing outward.
On the other hand, $\sum_{j \neq i} f_j( d_j(p)) v_j(p)$ is continuous at $q$ and hence it tends to a vector of fixed hyperbolic length.
Thus, for every $\varepsilon > 0$ there is a neighborhood $U$ of $q$ such that the angle between $v$ and $f(p)$ is at most $\varepsilon$, for any $p \in V$. This implies (a) in this case. If $I(q)= \{ j_1, \ldots, j_k\}$ is not a singleton and the inner unit normal vectors of $H_{j_1}, \ldots, H_{j_k}$ are denoted by $v_{j_1}, \ldots, v_{j_k}$, respectively, then a similar argument shows that if $p$ is `close to $q$', then $f(p)$ is `close' to the conic hull of these vectors.

Now we prove (b). Our method is to show that $q$ has a basis of closed neighborhoods with the property that no integral curve enters any of them, which clearly implies (b). For computational reasons we imagine the configuration in the Poincar\'e half space model, with $q$ as the `point at infinity'. The region $D$ in this model is the intersection of finitely many open hyperbolic half spaces with vertical and spherical bounding hyperplanes, where $H_i$ is vertical if, and only if $i \in I(q)$ (cf. Figure~\ref{fig:Thm4}).

\begin{figure}[ht]
\begin{center}
\includegraphics[width=0.9\textwidth]{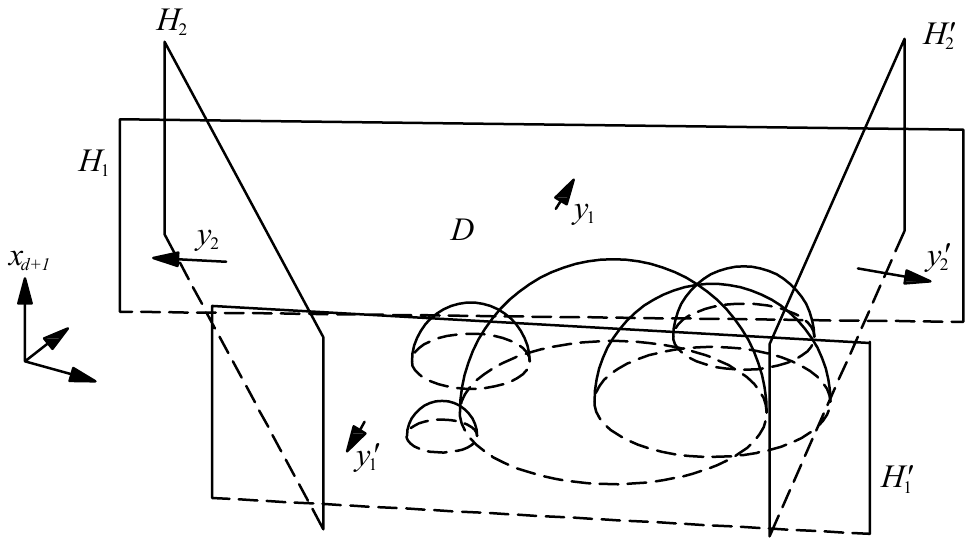}
\caption{The configuration in the Poincar\'e half space model}
\label{fig:Thm4}
\end{center}
\end{figure}

Consider a neighborhood $U$ of $q$. Then $U$ is the complement of a set which is bounded in $\Re^{d+1}$. Thus, without loss of generality, we may assume that $U$ is disjoint from all spherical $H_i$s, and it is bounded by a spherical hyperbolic hyperplane $H$.
For any $i \in I(q)$, let $y_i \in S$ be the outer unit normal vector of $H_i$ in $\Re^{d+1}$, where we set $S= \S^{d} \cap \{ x_{d+1} \}$.

Note that as $q$ is an ideal point of $D$, $D$ is not bounded in this model. Let $D^*$ denote the set of ideal points of $D$ on the Euclidean hyperplane $\{ x_{d+1} = 0\}$ (cf. Section~\ref{sec:prelim}). This set is the intersection of the closed half spaces $\bar{H}_i$, $i \in I(q)$ in the Euclidean $d$-space $\{ x_{d+1} = 0\}$ (for the definition of $H^*_i$, see Subsection~\ref{subsec:tools}).
Thus, if $D^*$ is not bounded, Lemma~\ref{lem:polytopes} implies that there are some closed vertical half spaces in $\HH^{d+1}$ whose intersection contains $H$, and whose outer unit normal vectors $y'_1, y'_2, \ldots, y'_m$ satisfy $\langle y'_j, y_i\rangle < 0$ for any $y_i$ and $y'_j$.
Let the intersection of these half spaces with $D$ be $D'$, and their bounding hyperbolic hyperplanes be $H'_1, H'_2, \ldots, H'_m$, where $y'_j$ is the outer unit normal vector of $H'_j$ for all values of $j$.

Let $p$ be a boundary point of $D'$ in $\HH^{d+1}$. Then $p \in H'_j \cap D'$ for some $j$s.
Observe that if $i \in I(q)$, then the geodesic line through $p$ and perpendicular to $H_i$, which in the model is a circle arc perpendicular to the hyperplane $\{ x_{d+1} \}$, is contained in the vertical plane through $p$ and perpendicular to $H_i$.
Thus, $v_i(p)$ points strictly inward into $D'$ at every boundary point of $D'$. A similar argument shows the same statement for any $i \notin I(q)$ as well.
As a result, we have that the integral curve through any point $p \in \bd D'$ enters $D'$ at $p$.

Let $X_t$ denote the set $\{ x_{d+1} = t \}$ for any $t > 0$, and note that this is a horosphere in $\HH^{d+1}$ with $q$ as its unique ideal point.
Set $D_t = X_t \cap D'$.
We show that if $t$ is sufficiently large, then $f(p)$ has a negative $x_{d+1}$-coordinate. We denote this coordinate by $z(p)$.

Let $p \in D_t$. It follows from Remark~\ref{rem:halfplane} and an elementary computation that if $i \in I(q)$, then the $x_{d+1}$-coordinate of
$v_i(p)$ is $\tanh d_i(p)$, and if $i \notin I(q)$, then it tends to $-1$ as $d_i(p) \to \infty$.
On the other hand, for any $\varepsilon , K > 0$ there is some value $t_0$ such that if $t > t_0$, then
$d_i(p) < \varepsilon$ for all $i \in I(q)$, and $d_i(p) > K$ for all $i \notin I(q)$ and for all $p \in D_t$.
This implies that
\[
\lim_{t \to \infty} \sup_{p \in D_t} z(p) = \lim_{d \to 0+0} \sum_{i \in I(q)} f_i(d) \tanh d - \lim_{d \to \infty} \sum_{i \notin I(q)} f_i(d).
\]
By the condition (\ref{eq:circle_condition}) and the relation (\ref{eq:angledistance}), we have that this quantity is negative, implying that
$z(p)$ is negative for all $p \in D_t$ if $t$ is sufficiently large.
Let $t'$ be chosen to satisfy this property. Without loss of generality, we may also assume that $X_{t'}$ does not intersect the hyperplane $H$.
Let $\bar{V}$ denote the set of points in $D'$ with $x_{d+1}$-coordinates less than $t'$, and let $V = \HH^{d+1} \setminus \bar{V}$.
Then $V$ is a neighborhood of $q$ in $\HH^{d+1}$, contained in $U$, and $V$ has the property that the integral curve through any boundary point $p$ of $V$
leaves $V$ at $p$. This proves (b).
\end{proof}

Now we finish the proof of Theorem~\ref{thm:circles}.
By the conditions in the formulation of the theorem, the set $\bigcup_{i=1}^n H_i \subset \HH^{d+1}$ is disconnected.
Let the components of this set be $X_1, X_2, \ldots, X_r$. By Lemma~\ref{lem:intcurves_ddim}, the integral curve of every point $p \in D$ terminates at some point of these sets. Let $Y_j$ denote the points of $D$ whose integral curve ends at a point of $X_j$. By Lemma~\ref{lem:intcurves_ddim}, no $Y_j$ is empty,
and it also implies that $Y_j$ is open in $D$ for all $j$s. Thus, $D$ is the disjoint union of the $r$ open sets $Y_1, Y_2, \ldots, Y_r$, where $r > 1$. On the other hand, $D$ is an open convex set, and thus, it is connected; a contradiction.

\section{Concluding remarks and open questions}\label{sec:remarks}

To illustrate why the problem of centering Koebe polyhedra via M\"obius transformations is different from the problem of centering density functions on the sphere, we prove Remark~\ref{rem:notdensity}.

\begin{rem}\label{rem:notdensity}
Let $g(\cdot) \in \{ \cc(\cdot), \cm_0(\cdot), \cm_1(\cdot), \cm_2(\cdot), \cm_3(\cdot)\}$ and let $P$ be a Koebe polyhedron. Then there is a M\"obius transformation $T : \S^2 \to \S^2$ such that $g(T(P)) \notin \BB^3$. Furthermore, if $P$ is simplicial, the same statement holds for $g(\cdot) = \ccm(\cdot)$.
\end{rem}

\begin{proof}
We use the notations introduced in Section~\ref{sec:prelim} and without loss of generality, we assume that the radius of every vertex and face circle of $P$ is less than $\frac{\pi}{2}$. Consider a closed hyperbolic half space $\bar{V}_i$ associated to a vertex circle $C_i$ of $P$. Let $q \in \S^2$ be an ideal point of $\bar{V}_i$ with the property that $q$ lies in the exterior of any vertex circle of $P$. Let the spherical distance of $q$ from the circle $C_i$ be $0 < \alpha \leq \frac{\pi}{2}$. Then the spherical radius of any vertex circle of $P$ is less than $\frac{\pi-\alpha}{2}$. Let $L$ be the hyperbolic line perpendicular to $V_i$ and with ideal point $q$. Consider a M\"obius transformation $T$ defined by a hyperbolic translation $T$ along $L$. Note that $T(q)=q$, and $T(V_i)$ is a hyperbolic plane perpendicular to $L$. Clearly, using a suitable translation, $T(V_i) \subset \inter \bar{V_i}$, and $T(V_i)$ is arbitrarily close to $o$. 
Here, the first property implies that, apart from $T(C_i)$, the radius of every vertex circle of $T(P)$ is less than $\frac{\pi-\alpha}{2} < \frac{\pi}{2}$, and hence,
by Remark~\ref{rem:fv} for any $j \neq i$ the distance of $j$th vertex of $T(P)$ from $o$ is less than $\frac{1}{\cos \frac{\pi-\alpha}{2}}$. On the other hand, the second property implies that the distance of the $j$th vertex of $T(P)$ from $o$ is arbitrarily large. From this, it readily follows that
$\cc(T(P)), \cm_0(T(P)), \cm_1(T(P)), \cm_2(T(P)), \cm_3(T(P)) \notin \BB^3$, and if $P$ is simplicial, then $\ccm(T(P)) \notin \BB^3$.
\end{proof}

We note that a similar construction can be given for spherical cap systems on $\S^d$ satisfying the conditions in Theorem~\ref{thm:circles}.

\begin{rem}
Using the idea of the proof in Subsection~\ref{subsec:euler}, it is possible to prove the following, stronger statement:
Let $P$ be a Koebe polyhedron, and let $g(\cdot) = \lambda_0 \cm_0(\cdot) + \lambda_1 \cm_1(\cdot) +  \lambda_2 \cm_2(\cdot) + \lambda_3 \cm_3(\cdot)$,
where $\sum_{i=0}^3 \lambda_i = 1$, $\lambda_i \geq 0$ for all values of $i$ and $\lambda_i > 0$ for some $i \neq 3$.
Then there is a M\"obius transformation $T$ such that $g(T(P))=o$. Furthermore, if $P$ is simplicial, the same statement holds for the convex combination
$g(\cdot) = \lambda_0 \cm_0(\cdot) + \lambda_1 \cm_1(\cdot) +  \lambda_2 \cm_2(\cdot) + \lambda_3 \cm_3(\cdot) + \lambda_4 \ccm(\cdot)$ under the same conditions.
\end{rem}

\begin{rem}
More elaborate computations, similar to those in Subsection~\ref{subsec:euler}, show that for any sufficiently large value of $t$,
the integral curves of the vector field $h_{cm}$ defined in (\ref{eq:cm}) cross $D_t$ in both directions.
This shows why our argument fails for $\cm_3(\cdot)$.
\end{rem}

\begin{rem}
An alternative way to prove Theorem~\ref{thm:main} for $\ccm(\cdot)$ seems to be the following. First, we triangulate the boundary of $P$ using the symmetric right trapezoids $Q_{i,j}$, or more specifically, we subdivide the faces by the incenters of the faces and the tangency points on the edges.
Computing the circumcenter of mass of this triangulation leads to a significantly simpler function for $\ccm(P)$ than the one in (\ref{eq:ccm_function}).
Nevertheless, as it is remarked in Subsection~\ref{subsec:centers}, circumcenter of mass is invariant only under triangulations that do not add new vertices to $\bd P$
(cf. also Remark 3.1 in \cite{Tabachnikov}).
\end{rem}

\begin{prob}\label{prob:balancing}
Prove or disprove that every combinatorial class of convex polyhedra contains a Koebe polyhedron whose center of mass is the origin.
\end{prob}

\begin{prob}
Prove or disprove that the M\"obius transformations in Theorem~\ref{thm:main} are unique up to Euclidean isometries.
\end{prob}

\begin{prob}
Is it possible to prove variants of Theorem~\ref{thm:circles} if the weight functions $w_i$ in (\ref{eq:vertices}) depend not only on $\rho_T(C_i)$ but also on the radii of the other spherical caps as well?
\end{prob}

\begin{prob}
Schramm \cite{Schramm} proved that if $K$ is any smooth, strictly convex body in $\Re^3$, then every combinatorial class of convex polyhedra contains a representative midscribed about $K$. If $K$ is symmetric to the origin, does this statement remain true with the additional assumption that the barycenter of the tangency points of this representative is the origin? Can the barycenter of the tangency points be replaced by other centers of the polyhedron?
\end{prob}

\section{Acknowledgment}
The author thanks G. Domokos, P. B\'alint, G. Etesi and Sz. Szab\'o for many fruitful discussions on this problem.

\end{document}